\documentclass{article}

\usepackage{graphicx}

\usepackage{amsmath}
\usepackage{amsfonts}
\usepackage{amssymb}
\usepackage{amsthm}

\usepackage{dsfont}
\usepackage{enumitem}

\usepackage[dvipsnames]{xcolor}
\definecolor{myred}{HTML}{c20014}
\definecolor{mygreen}{HTML}{008000}

\usepackage[colorlinks=true,linkcolor=myred,citecolor=mygreen]{hyperref}
\usepackage{url}

\usepackage{cleveref}
\crefname{equation}{}{}
\Crefname{equation}{Equation}{}

\providecommand{\keywords}[1]
{
  \small	
  \textbf{{Keywords.}} #1
}

\newcommand{\D}{\mathcal{D}}
\newcommand{\A}{\mathcal{A}}
\newcommand{\M}{\mathcal{M}}
\newcommand{\R}{\mathbb{R}}
\renewcommand{\S}{\mathcal{S}}
\newcommand{\on}{\omega_n}
\renewcommand{\H}{\mathcal{H}}
\renewcommand{\i}{\mathrm{i}}
\renewcommand{\d}{\mathrm{d}}
\newcommand{\eps}{\varepsilon}
\renewcommand{\L}{\mathcal{L}}

\DeclareMathOperator{\Id}{id}
\DeclareMathOperator{\re}{Re}
\DeclareMathOperator{\im}{Im}
\DeclareMathOperator{\dive}{div}
\DeclareMathOperator{\spe}{sp}

\newtheorem{theorem}{Theorem}
\newtheorem{proposition}{Proposition}

\title{Wave equation with hyperbolic boundary condition: a frequency domain approach
}
\author{Nicolas Vanspranghe\thanks{
Univ. Grenoble Alpes, CNRS, Grenoble INP, GIPSA-lab, 38000 Grenoble, France. Email: \nolinkurl{nicolas.vanspranghe@gipsa-lab.fr}.
}
}
\date{}

\begin{document}

\maketitle

\begin{abstract}
In this paper, we investigate the stability of the linear wave equation where one part of the boundary, which is seen as a lower-dimensional Riemannian manifold, is governed by a coupled wave equation, while the other part is subject to a dissipative Robin velocity feedback. We prove that the closed-loop equations generate a semi-uniformly stable semigroup of linear contractions on a suitable energy space.
Furthermore, under multiplier-related geometrical conditions, we establish a polynomial decay rate for strong solutions. This is achieved by estimating the growth of the resolvent operator on the imaginary axis.
\end{abstract}

\keywords{
Infinite-dimensional systems, control of partial differential equations, control and estimation of wave equations and systems of elasticity, semigroup and operator theory. 
}

\section{Introduction and main results}

\subsection{Background}

Let $\Omega$ be a  bounded domain of $\mathbb{R}^d$, $d\geq 2$, with smooth boundary $\Gamma = \Gamma_0 \cup \Gamma_1$. We assume that $\Gamma_0$ and $\Gamma_1$ are relatively open non-empty subsets of $\Gamma$ that satisfy $\overline{\Gamma_0} \cap \overline{\Gamma_1} = \emptyset$. We consider the following feedback system:
\begin{subequations}
\label{eq:pde-bc}
\begin{align}
\label{eq:pure-wave}
&\partial_{tt}u - \Delta u = 0 & &\mbox{in}~ \Omega \times (0, +\infty), \\
\label{eq:hbc}
&\partial_{tt}u  - \Delta_\Gamma u = - \partial_\nu u & &\mbox{on}~ \Gamma_0 \times (0, +\infty), \\
\label{eq:robin-feed}
&\partial_\nu u + u= - \alpha \partial_t u & & \mbox{on}~ \Gamma_1 \times (0, + \infty),
\end{align}
\end{subequations}
where $\alpha$ is a positive constant, $\Delta$ is the Laplacian, $\partial_\nu$ denotes the outward normal derivative, and $\Delta_\Gamma$ is the Laplace-Beltrami operator on $\Gamma$ for the metric inherited from $\R^d$ (see Subsection \ref{sec:prel} below).

The general context of this work is the analysis  of evolution equations with dynamic (or kinetic) boundary conditions. 
Those arise in physical models  where the momentum of the boundary cannot be neglected, hence the second-order (in time) dynamics.
%
An early example of such equations is given by \cite{liu1998spectral}, where energy decay of a two-dimensional (in space) acoustic flow is studied.
In our case, the coupled wave equation \cref{eq:hbc} may model boundary oscillations that propagate in the tangential directions and are caused by in-domain displacements governed by the pure wave equation \cref{eq:pure-wave}. 
 A few variations around the coupled equations \cref{eq:pure-wave}-\cref{eq:hbc} have been investigated in the literature, with \cref{eq:robin-feed} being typically replaced by a zero Dirichlet boundary condition. \cite{vitillaro2017wave} deals with local and global well-posedness of \cref{eq:pure-wave}-\cref{eq:hbc} perturbated by nonlinear potentials and damping terms acting on the domain and the boundary. In \cite{graber_analicity}, \cref{eq:pure-wave}-\cref{eq:hbc} are supplied with boundary and/or in-domain Kelvin-Voigt damping, which adds heat-like regularizing effect to the flow. 
The present article is more control-oriented and tackles the problem of boundary stabilization of \cref{eq:pure-wave}-\cref{eq:hbc} by the mean of a velocity feedback acting on $\Gamma_1$ only, as modeled by \cref{eq:robin-feed}. To the best of our knowledge, this problem has not been addressed.
Overall, what differentiates our work from the related literature is  the combination of the  two following technical challenges.
\begin{enumerate}
\item
In presence of the Laplace-Beltrami term, the boundary condition \cref{eq:hbc} is a proper (hyperbolic) partial differential equation, as opposed to \cite{liu1998spectral} or the recent article by  \cite{li_asymptotics} for instance, where no tangential derivatives appear in the dynamic boundary condition.

\item Only the anticollocated boundary $\Gamma_1$ dissipates energy; in other words, from the point of view of the dynamic boundary $\Gamma_0$, the damping is indirect and has to somehow propagate across $\Omega$. This contrasts with all the aforementioned work, where damping acts in the interior and/or the boundary subject to the second-order dynamics.
\end{enumerate}

Inspired by the literature on coupled second-order equations and in particular \cite{liu_frequency_2007}, we carry out the stability analysis of the feedback system \cref{eq:pde-bc} in the frequency domain: we investigate  pure imaginary eigenvalues (or rather, the lack thereof) and then  aim at estimating the growth of the resolvent operator on the imaginary axis.
By doing so, we are able to prove semi-uniform stability of system \cref{eq:pde-bc} and, under additional geometrical conditions, polynomial energy decay for solutions with smooth initial data. This is detailed in the next subsection.
Finally, let us also mention \cite{alabau_indirect}, where polynomial stability is established for a class of abstract coupled second-order equations; however, this result does not apply to \cref{eq:pde-bc} due to the unboundedness of the corresponding coupling operator.
In particular, the compact perturbation argument, which is often employed to prove that weakly damped systems of waves are \emph{not} uniformly stable, cannot be used, leaving the question of {exponential} stability open.
%


\textbf{Notation.} The norm of a given normed vector space $E$ is denoted by $\|\cdot \|_E$. The duality bracket $\langle \phi, x \rangle_{E}$ is used to write $\phi (x)$ for any vector $x$ in $E$ and continuous linear form $\phi$ in $E'$. If $E$ is a Hilbert space, then $(\cdot, \cdot)_E$ denotes the scalar product of $E$. If $E_1$ and $E_2$ are two Banach spaces, $\L(E_1, E_2)$ denotes the set of bounded linear operators from $E_1$ to $E_2$, which is a Banach space as well if equipped with the operator norm. Given a real number $s$, we denote by $H^s(\Omega)$ the (complex) Sobolev space of order $s$ on $\Omega$. The notation $\d x$ indicates the Lebesgue measure on $\R^d$; and $\d \sigma$ denotes the induced surface measure on $\Gamma$. Finally, $\mathcal{C}_c^\infty(\Omega)$ is the space of compactly supported and infinitely differentiable complex-valued functions on $\Omega$. In the proofs, $K$, $K'$, etc., stand for generic constants that do not depend on the variables of interest.

\subsection{Main statements}
We start by introducing the natural energy space $\H$ associated with the feedback system \cref{eq:pde-bc}. Let
 \begin{equation}
H \triangleq L^2(\Omega) \times L^2(\Gamma_0)
\end{equation}
endowed with its product Hilbertian structure, and
\begin{equation}
V \triangleq \{ (u,  \theta) \in H^1(\Omega) \times H^1(\Gamma_0) : u_{|\Gamma_0} =  \theta \}
\end{equation}
equipped with a scalar product $(\cdot, \cdot)_V$ explicitly defined below in \cref{eq:scalar-product} and equivalent to that of ${H^1(\Omega)\times H^1(\Gamma_0)}$. The set $V$ is a Hilbert space as well (see Subsection \ref{sec:prel} below). Then, we define the product Hilbert space
\begin{equation}
\H \triangleq V \times H.
\end{equation}
Our first result concerns well-posedness in $\H$ and semi-uniform stability of the system governed by \cref{eq:pde-bc}.  We start by recasting the boundary value problem \cref{eq:pde-bc} into a first-order evolution equation on $\H$ of the form $(\d /\d t)[u, v] + \A [u, v] = 0 $, where $\A : \D(\A) \to \H$ is an  unbounded linear operator  explicitly given below in \cref{eq:generator}. Solutions to \cref{eq:pde-bc} are understood in the usual linear semigroup sense: they are \emph{classical} solutions for initial data in the domain $\D(\A)$ and \emph{mild} solutions for general initial data in $\H$.
\begin{theorem}
\label{th:wp}
Solutions  to \cref{eq:pde-bc} define a strongly continuous semigroup $\{ \S_t \}$ of linear contractions on the energy space $\H$, with maximal dissipative generator $- \A$. Furthermore, $\{ \S_t\}$ is semi-uniformly stable, i.e., $\{\S_t\}$ is bounded and
\begin{equation}
\label{eq:semi-uniform}
\lim_{t \to + \infty} \|\S_t (\A + \Id)^{-1}\|_{\L(\H)} = 0.
\end{equation}
\end{theorem}
The proof of Theorem \ref{th:wp} is given in Section \ref{sec:wp}.
We digress for a moment to comment on the notion of \emph{semi-uniform} stability, which has been introduced in \cite{batty-non-uniform}.
As the name suggests, it is a property that is intermediate between strong and uniform stability. Indeed, \cref{eq:semi-uniform} implies that $\{\S_t\}$ is strongly stable
and that  the decay of \emph{strong} solutions to \cref{eq:pde-bc} can be quantified as follows:
\begin{equation}
\label{eq:semi-uniform-bis}
\|\S_t [u_0, v_0]\|_\H \leq K \|\S_t(\A + \Id)^{-1}\|_{\L(\H)} \|[u_0, v_0]\|_{\D(\A)}
\end{equation}
for any initial data $[u_0, v_0]$ in $\D(\A)$ equipped with the graph norm. For more details, the reader is referred to the survey article by \cite{chill-semi}.  As an example, semi-uniform stability of a wave equation with spatially varying coefficients is investigated using spectral methods in \cite{jacob_stability}.
 Coming back to our contributions, under certain geometrical conditions, we are able to replace \cref{eq:semi-uniform-bis} with an explicit polynomial decay rate.
\begin{theorem}
\label{th:pdr}
Assume  there exists a real vector field $h$ in $\mathcal{C}^2(\overline{\Omega})$ that satisfies the following conditions:
\begin{enumerate}[label={(\alph*)}]
\item  
\label{it:jacobian} 
Denoting the Jacobian matrix of $h$ by $J_h \triangleq [\partial_j h_i]_{i j}$, there exists $\rho > 0$ such that
\begin{equation}
\re \int_\Omega [J_h f] \cdot \overline{f} \, \d x \geq \rho \|f\|^2_{L^2(\Omega)^d}
\end{equation}
for all $f$ in $L^2(\Omega)^d$;
\item
\label{it:gamma0}
On $\Gamma_0$, $h$ is parallel to the unit outward normal $\nu$, i.e., $h = (h \cdot \nu) \nu$; also, $h \cdot \nu \leq 0$;
\item
\label{it:gamma1}
 On $\Gamma_1$,  $(h \cdot \nu) \geq m$ for some $m > 0$.
\end{enumerate}
Then, the semigroup $\{ \S_t\}$ enjoys the following polynomial decay property:
there exists $C > 0$ such that for any $[u_0, v_0]$ in $\D (\A)$, for all $t\geq 0$,
\begin{equation}
\label{eq:pdr}
\|\S_t[u_0, v_0]\|_\H \leq C t^{-{1/2}} \|[u_0, v_0]\|_{\D(\A)}.
\end{equation}
%
\end{theorem}
Theorem \ref{th:pdr} is proved in Section \ref{sec:pdr}.
Most of its geometrical requirements are standard when it comes to
differential multiplier analysis; we point out however that Item \ref{it:gamma0} is  a stronger than usual assumption in that we use a vector field that is perpendicular to the boundary on $\Gamma_0$.
Nevertheless, examples of such domains include ``donut-shaped" sets $\Omega$ of the form 
$
\Omega = \{ x \in \R^d : k_0 < f(x) < k_1 \}
$
where $f : \R^d \to \R$ is a smooth strictly convex function, and $k_0$ and $k_1$ are real numbers such that $k_0 < k_1$ with $k_0 > \inf_{x \in \R^d} f(x)$. In that case, $\Gamma_0$ and $\Gamma_1$ are the inverse image by $f$ of $\{k_0\}$ and $\{k_1\}$ respectively, and one can check the hypotheses of Theorem \ref{th:pdr} by letting $h = \nabla f$.

\subsection{Preliminaries and operator model}
\label{sec:prel}

In this subsection, we introduce additional definitions and notation that are needed in our analysis of system \cref{eq:pde-bc}.

The boundary $\Gamma$ of the domain $\Omega$ is a compact and smooth  embedded submanifold of the ambient Euclidian space $\R^d$.
Recalling \cite[Chapitre 1, Section 7.3]{lions-problemes}, the Sobolev spaces $H^s(\Gamma)$ are modeled after  $H^s(\R^{d-1})$ by the mean of partitions of unity subordinated to the covering of $\Gamma$ by charts. 

For each $x$ in $\Gamma$, we denote by $T_x(\Gamma)$ the tangent space at $x$, which we see as a $(d-1)$-dimensional subspace of $\R^d$. Given a smooth function $\varphi : \Gamma \to \R$
 , the total derivative of $\varphi$ at $x \in \Gamma$, which is a linear form on $T_x(\Gamma)$, is denoted  by $\d \varphi(x)$ -- see for instance \cite[Chapter 1]{guillemin-topology}. As a submanifold, $\Gamma$ can be equipped with the canonical Riemannian metric $g$ inherited from $\R^d$:
$
g_x(\gamma_1, \gamma_2) = \gamma_1 \cdot \gamma_2
$
for all $\gamma_1, \gamma_2 \in T_x(\Gamma)$ and $x  \in \Gamma$, where $\cdot$ denotes the usual Euclidian inner product. The Riemannian measure associated with $g$ coincides with the induced hypersurface measure $\d \sigma$.  The Riemannian gradient $\nabla_\Gamma \varphi$ of a smooth real-valued  function $\varphi$ is defined as follows: $\nabla_\Gamma \varphi(x)$ is the unique element in $T_x(\Gamma)$ such that $\d \varphi(x) \gamma = (\nabla_{\Gamma}\varphi \cdot \gamma)_\Gamma$ for all $\gamma$ in $T_x(\Gamma)$. Then, $\nabla_\Gamma \varphi$ is a smooth vector field on $\Gamma$. This definition extends to complex-valued $\varphi$ by linearity. Following \cite[Chapter 2]{taylor-pde}, the Laplace-Beltrami operator $\Delta_\Gamma$ is defined to be the second-order differential operator on $\Gamma$ satisfying
$
\label{eq:def-laplace}
-\int_{\Gamma} \Delta_\Gamma \varphi_1 \varphi_2 \, \d \sigma = \int_{\Gamma} \nabla_\Gamma \varphi_1 \cdot \nabla_\Gamma \varphi_2 \, \d \sigma
$
for all smooth and compactly supported $\varphi_1$ and $\varphi_2$.
One can then define $\nabla_\Gamma \theta$ and $\Delta_\Gamma \theta$ in the sense of distributions for any $\theta$ in (say) $L^2(\Gamma)$. Then, $H^1(\Gamma)$ is the set of all $\theta$ in $L^2(\Gamma)$ such that $\nabla_\Gamma \theta$ belongs to $L^2(\Gamma)^{d}$.
 (recall that here each $T_x(\Gamma)$ is a subspace of $\R^d$). 
 Using the notation $\|x\|^2 \triangleq x\cdot \overline{x}$ for $x$ in $\mathbb{C}^d$, the norm  given by
$
\|\theta\|_{H^1(\Gamma)}^2 = \int_{\Gamma} |\theta|^2 + \| \nabla_\Gamma \theta \|^2 \, \d \sigma
$
is equivalent to those built upon local charts. Likewise, $H^2(\Gamma)$ is the space of all $\theta$ in $L^2(\Gamma)$ such that $ - \Delta_\Gamma \theta$ belongs to $L^2(\Gamma)$.
 For more details, the reader is referred to \cite[Chapters 4 and 5]{taylor-pde}.

From now on, we focus on the submanifold $\Gamma_0$. It follows from the assumption $\overline{\Gamma_0} \cap \overline{\Gamma_1} = \emptyset$ that $\Gamma_0$ is connected and has no boundary. Thus, the spaces $H^1_0(\Gamma_0)$ and $ H^1(\Gamma_0)$ coincide; and for any real $s$, $-\Delta_\Gamma$ extends as a bounded linear operator from $H^s(\Gamma_0)$ to $H^{s - 2}(\Gamma_0)$.
Furthermore, we have the following Green formula on $\Gamma_0$:
\begin{equation}
\label{eq:div-manifold}
\int_{\Gamma_0} \nabla_\Gamma \theta_1 \cdot \nabla_\Gamma \theta_2\, \d \sigma = - \int_{\Gamma_0} \Delta_\Gamma \theta_1 \theta_2 \, \d \sigma,
\end{equation}
for any $\theta_1$ in $H^2(\Gamma_0)$ and $\theta_2$ in $H^1(\Gamma_0)$. Finally, we recall that for sufficiently smooth $u$, say, $u \in H^2(\Omega)$,
the vector field given by the \emph{tangential} derivatives of $u$ on $\Gamma_0$
coincides with the Riemannian gradient $\nabla_\Gamma u$  of the trace $u_{|\Gamma_0}$.
This allows us to write
\begin{equation}
\label{eq:tan-riem}
\|\nabla u \|^2 = |\partial_\nu u|^2 + \|\nabla_\Gamma u \|^2 \quad \mbox{a.e. on}~ \Gamma_0.
\end{equation}

Let us return to the spaces $H$ and $V$. 
One can prove that $V$ is closed
 in $H^1(\Omega) \times H^1(\Gamma_0)$, which makes it a Hilbert space if equipped with the inherited scalar product. In the sequel, we will rather use the following one:
\begin{equation}
\label{eq:scalar-product}
(u_1, u_2)_V \triangleq \int_\Omega \nabla u_1 \cdot \nabla \overline{u_2} \, \d x + \int_{\Gamma_0} \nabla_\Gamma u_1 \cdot \nabla_\Gamma \overline{u_2} \, \d \sigma   + \int_{\Gamma_1} u_1 \overline{u_2} \, \d \sigma.
\end{equation}
Using a standard indirect compactness argument, we see that the norm associated with \cref{eq:scalar-product} is equivalent to that of $H^1(\Omega) \times H^1(\Gamma_0)$. Note that we will frequently identify $V$
as a subspace of $H^1(\Omega)$ and drop the tuple notation. 
We can finally define the operator $\A$: let $W \triangleq [H^2(\Omega) \times H^2(\Gamma_0)] \cap V$, then
\begin{subequations}
\label{eq:generator}
\begin{align}
&\D(\A)  \triangleq \{ [u, v] \in W \times V: \partial_\nu u + u = - \alpha v ~\mbox{on}~ \Gamma_0 \}, \\
&\A[u, v] \triangleq [- v, ( - \Delta u, - \Delta_\Gamma u + \partial_\nu u)].
\end{align}
\end{subequations}

\section{Well-posedness and semi-uniform stability}
\label{sec:wp}

To prove Theorem \ref{th:wp}, we first investigate properties of $\A$.
\begin{proposition}
\label{prop:max-mon}
The unbounded operator $\A$ is maximal monotone. Furthermore, for any $\lambda > 0$, the resolvent  $(\A + \lambda \Id)^{-1}$ is a compact operator on $\H$.
\end{proposition}
\begin{proof} The proof is split into several steps. 


\textbf{Step 1: Monotonicity.}
 Let $X = [u, v] \in \D(A)$. By performing a few integration by parts, we obtain the following formula:
\begin{equation}
\label{eq:ax-x}
(\A X, X)_\H =  \alpha \int_{\Gamma_1} |v|^2 \, \d \sigma + 2\i \im (u, v)_V.
\end{equation}
Taking the real part of \cref{eq:ax-x} yields $\re (\A X, X)_\H \geq 0$.

\textbf{Step 2: Variational equations.} Let $\lambda > 0$. Our goal is to prove that $\A + \lambda \Id$ is surjective. We will simultaneously prove that $(\A + \lambda \Id)^{-1}$ is well-defined and compact. Let $[f, g] \in \H$ with $g = (g_1, g_2) \in H$. We need to find $X = [u, v] \in \D(\A)$ such that $\A X + X = [f, g]$, i.e., $- v + \lambda u = f$ and
\begin{subequations}
\label{eq:lambda-problem}
\begin{align}
\label{eq:u-v-g1}
&- \Delta u + \lambda v = g_1 &&\mbox{in}~ \Omega, \\
\label{eq:u-v-g2}
&- \Delta_\Gamma u + \partial_\nu u + \lambda v = g_2 &&\mbox{on}~\Gamma_0.
\end{align}
\end{subequations}
We infer from \cref{eq:u-v-g1}-\cref{eq:u-v-g2} that any solution $[u, v]$ must satisfy the following variational problem:
\begin{multline}
\label{eq:var-u}
\int_\Omega  \nabla u \cdot \nabla \overline{w}  \, \d x + \int_{\Gamma_0} \nabla_\Gamma u  \cdot \nabla_\Gamma \overline{w} \, \d \sigma + \int_{\Gamma_1} [u + \alpha v] \overline{w} \, \d x \\ + \lambda (v, w)_H = (g, w)_H \quad \mbox{for all}~ w \in V.
\end{multline}
As usual for that kind of problem (see for instance \cite[Proof of Proposition 2.1]{liu_frequency_2007}), the existence of $[u, v] \in \H$ satisfying both  $-v + \lambda u = f$ and \cref{eq:var-u} is proved by obtaining a variational equation in the $v$-variable only, and then using Lax-Milgram theorem to find an appropriate $v \in V$, which in turn uniquely determines $u$. It remains to prove that $[u, v]$ belongs to $\D(\A)$ and that \cref{eq:u-v-g1}-\cref{eq:u-v-g2} are  satisfied in a $L^2$-sense. By evaluating \cref{eq:var-u} for test functions $w$ in $\mathcal{C}_c^\infty(\Omega)$, we obtain that the distribution $\Delta u$ is in fact a function in $L^2(\Omega)$, with
 \cref{eq:u-v-g1} satisfied a.e. in $\Omega$. Recall that $\partial_\nu u$ is then uniquely defined in $H^{-1/2}(\Gamma)$ by the formula
\begin{equation}
\label{eq:distrib-normal}
\langle \partial_\nu u , \theta \rangle_{H^{1/2}(\Gamma)} = \int_{\Omega} \nabla u \cdot \nabla w \, \d x - \int_{\Omega} \Delta u w \, \d x
\end{equation}
for all $\theta \in H^{1/2}(\Gamma)$, where $w$ is any element in $H^1(\Omega)$ such that $w_{|\Gamma} = \theta$. Furthermore,
\begin{equation}
\|\partial_\nu u\|_{H^{-1/2}(\Gamma)} \leq K \{ \|\Delta u\|_{L^2(\Omega)} + \|u\|_{H^1(\Omega)}\}.
\end{equation}
Plugging \cref{eq:u-v-g1} and \cref{eq:distrib-normal} into \cref{eq:var-u} leads to another variational equation, from which we shall recover the boundary conditions satisfied by $u$:
\begin{multline}
\label{eq:var-gamma}
 \langle -\Delta_{\Gamma}u, \overline{w}_{|\Gamma_0} \rangle_{H^1(\Gamma_0)} + \langle  \partial_\nu u, \overline{w}_{|\Gamma}\rangle_{H^{1/2}(\Gamma)}  \\  -  \int_{\Gamma_1} [u + \alpha v] \overline{w} \, \d \sigma = \int_{\Gamma_0} g_2 \overline{w} \, \d \sigma \quad \mbox{for all}~ w \in V.
\end{multline}

\textbf{Step 3: ``Decoupling" the boundary conditions.}
Since $\overline{\Gamma_0} \cap \overline{\Gamma_1} = \emptyset$, the indicator functions $\mathds{1}_{\Gamma_0}$ and $\mathds{1}_{\Gamma_1}$ are smooth. As a notable consequence, for any real $s \geq 0$,  the extension map  $\theta \to \mathds{1}_{\Gamma_i} \theta$ belongs to $\L(H^s(\Gamma_i), H^s(\Gamma))$, $i \in \{0, 1\}$. In particular, given an arbitrary $\theta \in H^1(\Gamma_0)$, $\mathds{1}_{\Gamma_0} \theta$ is in 
$H^{1/2}(\Gamma)$,
so that taking any continuous right-inverse of the trace provides an element $w \in V$ satisfying $w_{|\Gamma_0} = \theta$ and $w_{|\Gamma_1} = 0$. Evaluating \cref{eq:var-gamma} for such $w$ yields
\begin{equation}
\label{eq:var-gamma0}
 \langle -\Delta_{\Gamma}u, \theta \rangle_{H^1(\Gamma_0)} + \langle \partial_\nu u, \mathds{1}_{\Gamma_0} \theta \rangle_{H^{1/2}(\Gamma)}  
= \int_{\Gamma_0} [g_2 - \lambda v] \theta \,  \d \sigma
\end{equation}
holding for arbitrary $\theta \in H^1(\Gamma_0)$.  Again,  the map $\theta \mapsto \mathds{1}_{\Gamma_0} \theta$ is in  $\L(H^{1/2}(\Gamma_0), H^{1/2}(\Gamma))$; hence, it follows from \cref{eq:var-gamma0} that $- \Delta_{\Gamma}u$, which is \emph{a priori} defined in $H^{-1}(\Gamma_0)$, belongs in fact to $H^{-1/2}(\Gamma_0)$. 
Then, elliptic regularity for the Laplace-Beltrami operator -- see, e.g., \cite{taylor-pde} -- yields $u_{|\Gamma_0} \in H^{3/2}(\Gamma_0)$ and
\begin{multline}
\label{eq:3/2-gamma0}
\|u_{|\Gamma_0}\|_{H^{3/2}(\Gamma_0)} \leq K \{ \|\partial_\nu u\|_{H^{-1/2}(\Gamma)} +  \|g_2- \lambda v\|_{L^2(\Gamma_0)}\} \\
 \leq K' \{ \|\Delta u \|_{L^2(\Omega)} + \|u\|_{H^1(\Omega)} + \|g_2 - \lambda v\|_{L^2(\Gamma_0)}\} \\ \leq K' \{  \|g_1 - \lambda v\|_{L^2(\Omega)} + \|u\|_{H^1(\Omega)} + \|g_2- \lambda v\|_{L^2(\Gamma_0)}\}.
\end{multline}
To prove that $u \in H^2(\Omega)$, we start by picking a function $\rho \in \mathcal{C}^2(\overline{\Omega})$ such that $\rho = 1$ (resp. $\rho = 0$) in some open  neighborhood $\Gamma_0^\eps$ of $\Gamma_0$ (resp. $\Gamma_1^\eps$ of $\Gamma_1$). We let $u^0 \triangleq \rho u$ and $u^{1} \triangleq (1- \rho) u$, so that $u = u^0 + u^1$. Then, $u^0$ belongs to $H^1(\Omega)$ and $\Delta u^0 = \rho \Delta u + u \Delta \rho + 2 \nabla u \cdot \nabla \rho \in L^2(\Omega)$. First, we have $u^1_{|\Gamma} = \mathds{1}_{\Gamma_0} u_{|\Gamma} \in H^{3/2}(\Gamma_0)$. Then, applying elliptic theory (\cite{lions-problemes,taylor-pde}), we get that  $u^1 \in H^2(\Omega)$ together with the estimate
\begin{multline}
\label{eq:elliptic-u0}
\|u^0\|_{H^2(\Omega)} \leq K \{ \|\Delta u^0\|_{L^2(\Omega)} + \|u_{|\Gamma_0}\|_{H^{3/2}(\Gamma_0)}\} \\
\leq K' \{  \|g_1 - \lambda v\|_{L^2(\Omega)} + \|u\|_{H^1(\Omega)}+  \|g_2- \lambda v\|_{L^2(\Gamma_0)}\}.
\end{multline}
Next, we look at $u^1$. Again, $\Delta u^1 \in L^2(\Omega)$, and $\partial_\nu u^1$ is well-defined in $H^{-1/2}(\Gamma)$ as well. 
We claim that the (distributional) normal derivative $\partial_\nu u^1$
satisfies, for any $\theta \in H^{1/2}(\Gamma)$, 
$
\langle \partial_\nu u^1, \theta \rangle_{H^{1/2}(\Gamma)} = \langle \partial_\nu u, \mathds{1}_{\Gamma_1} \theta \rangle_{H^{1/2}(\Gamma)}.
$
This can be deduced from \cref{eq:distrib-normal} by constructing, given $\theta \in H^{1/2}(\Gamma)$, a function $w \in H^1(\Omega)$ satisfying $w_{|\Gamma} = \mathds{1}_{\Gamma_1} \theta$ and whose support is contained in $\Gamma_1^{\eps}$ ( where $u$ and $u^1$ coincide).
On the other hand, using the same argument as for $\Gamma_0$, we can particularize \cref{eq:var-gamma} to elements $w \in V$ satisfying $w_{|\Gamma} = \mathds{1}_{\Gamma_1}\theta $ for any given arbitrary $\theta \in H^{1/2}(\Gamma)$.
 This leads to
\begin{equation}
\langle \partial_\nu u, \mathds{1}_{\Gamma_1} \theta \rangle_{H^{1/2}(\Gamma)} + \int_{\Gamma_1} [u + \alpha v] \theta \, \d \sigma = 0,
\end{equation}
which means that $\partial_\nu u^1$ is  in $H^{1/2}(\Gamma)$, with $\partial_\nu u^1 = \mathds{1}_{\Gamma_1}[- u - \alpha v]$. Combined with $\Delta u^1 \in L^2(\Omega)$, elliptic theory yields that $u^1 \in H^2(\Omega)$, with
\begin{equation}
\label{eq:elliptic-u1}
\|u^1\|_{H^2(\Omega)} \\ \leq K \{ \|g_1 - \lambda v\|_{L^2(\Omega)} +  \|u\|_{H^1(\Omega)} + \|v\|_{H^1(\Omega)}\}.
\end{equation}
Then, $u = u^0 + u^1 \in H^2(\Omega)$, $\partial_\nu + u = - \alpha v$ on $\Gamma_1$. Going back to \cref{eq:var-gamma0}, we see that $\Delta_\Gamma u \in L^2(\Gamma_0)$ and thus $u_{|\Gamma_0} \in H^2(\Gamma_0)$. It is now proved that $[u, v] \in \D(\A)$. 
%

\textbf{Step 4: Compactness.} The following argument is standard: by substituting the identity $-v + \lambda u = g_1$ in order to
rewrite the variational problem \cref{eq:var-u} in terms of $u$ only and letting $w = u$ in the resulting equation, one can obtain an estimate of the form $\|[u, v]\|_\H \leq K \|[f, g]\|_H$. From there, we combine \cref{eq:3/2-gamma0}, \cref{eq:elliptic-u0}, and \cref{eq:elliptic-u1} to obtain
\begin{equation}
\label{eq:stronger-estimate}
\|u\|_{H^2(\Omega) \times H^{3/2}(\Gamma_0)} + \|v\|_V \leq K \|[f, g]\|_\H.
\end{equation}
Since $H^2(\Omega) \times H^{3/2}(\Gamma_0)$ is compactly embedded into $H^1(\Omega) \times H^1(\Gamma_0)$, and $V$ is compactly embedded into $H$, \cref{eq:stronger-estimate} proves that $(\A + \lambda \Id)^{-1}$ is a compact operator.
\end{proof}

The next proposition is motivated by the spectral criterion for semi-uniform stability.
\begin{proposition}
\label{prop:spectrum}
 We have $ \spe(\A) \cap \i \R = \emptyset$.
\end{proposition}
\begin{proof}
First, 
due to the compactness of $(\A + \lambda \Id)^{-1}$ for $\lambda > 0$,  $\spe(\A)$ consists of only eigenvalues. That being said, we now prove the result by contradiction. Suppose there exists $\lambda = \mathrm{i}\omega \in \mathrm{i}\mathbb{R}$ such that for some non-zero $X = [u, v] \in \D(\A)$, $\A X = \i \omega X$. We start with the case $\omega = 0$. Then, $\A[u, v] = 0$, which means that $v = 0$ and $u$ solves the following boundary-value problem:
\begin{subequations}
\begin{align}
\label{eq:harm}
&- \Delta u = 0 &&\mbox{in}~ \Omega, \\
\label{eq:harm-gamma0}
&- \Delta_\Gamma u = - \partial_\nu u  &&\mbox{on}~ \Gamma_0, \\
\label{eq:harm-gamma1}
& \partial_\nu u + u = 0 & &\mbox{on}~ \Gamma_1.
\end{align}
\end{subequations}
We multiply \cref{eq:harm} by $\overline{u}$, integrate over $\Omega$, and use \cref{eq:harm-gamma0}-\cref{eq:harm-gamma1} along with Green formulas on $\Omega$ and $\Gamma_0$ to obtain
 $\|u\|^2_V = 0$; thus, $X = 0$. Now, in the case where  $\omega$ is non-zero, we can write
\begin{equation} 
\|X\|^2_\H = \frac{1}{\i \omega} (\A X, X)_\H.
\end{equation}
Recalling the identity \cref{eq:ax-x}, we have
\begin{equation}
\label{eq:norm-id}
\|X\|^2_\H = \frac{\alpha}{\i \omega} \int_{\Gamma_1} |v|^2 \, \d \sigma - \frac{2}{\omega} \im (u, v)_V.
\end{equation}
Taking the imaginary part of \cref{eq:norm-id} yields 
$
v = 0
$ a.e. on $\Gamma_1$.
On the other hand, $u = - \i \omega v$ and because $[u, v] \in \D(\A)$,
\begin{subequations}
\begin{align}
&- \Delta u - \omega^2u = 0 &&\mbox{in}~ \Omega, \\
& u = 0 && \mbox{on}~ \Gamma_1, \\
& \partial_\nu u = 0 & &\mbox{on}~ \Gamma_1.
\end{align}
\end{subequations}
Furthermore, the differential operator $- \Delta - \omega^2 \Id$ is elliptic and has real analytic coefficients. Thus, we can apply {John-Holmgrem theorem} on unique continuation across  non-characteristic hypersurfaces
 to obtain that $u = 0$ in $\Omega$, which completes the proof.
\end{proof}

We now conclude the section.

\begin{proof}[Proof of Theorem \ref{th:wp}] As a maximal dissipative operator, $-\A$ generates a strongly continuous semigroup of linear contractions on $\H$ by virtue of Lumer-Phillips theorem. Furthermore, because $\spe(\A) \cap \i \R = \emptyset$, we can apply \cite[Theorem 1]{batty-non-uniform} to obtain the desired semi-uniform stability property \cref{eq:semi-uniform}.
\end{proof}

\section{Resolvent estimate and polynomial decay rate}
\label{sec:pdr}

The main technical contribution of our paper is the following resolvent estimate.
\begin{proposition}
\label{prop:resolvent-estimate}
Under the geometrical conditions of Theorem \ref{th:pdr}, we have
\begin{equation}
\label{eq:resolvent-estimate}
\sup_{\omega \in \mathbb{R},  |\omega| \geq 1} \frac{1}{\omega^2}  \left \|  (\mathcal{A} + \mathrm{i} \omega \mathrm{id} )^{-1} \right \|_{\mathcal{L}(\mathcal{H})} < + \infty.
\end{equation}
\end{proposition}
Assume for a moment that Proposition \ref{prop:resolvent-estimate} is established. 

\begin{proof}[Proof of Theorem \ref{th:pdr}] We recall that $\{\S_t\}$ is a bounded semigroup with generator $-\A$ that satisfies $\spe(\A) \cap \i \R = \emptyset$. Thus, we can
apply \cite[Theorem 2.4]{borichev_polynomial} to deduce from \cref{eq:resolvent-estimate} that for some $C > 0$,
\begin{equation}
\|\S_t (\A + \Id)^{-1} \|_{\L(\H)} \leq C t^{-1/2}.
\end{equation}
The operator $(\A + \Id)^{-1}$ is an isomorphism between $\H$ and $\D(\A)$ endowed with the graph norm, hence \cref{eq:pdr}.
\end{proof}

We now give the proof of the desired resolvent estimate.

\begin{proof}[Proof of Proposition \ref{prop:resolvent-estimate}]
We proceed by contradiction. Assume there exist  sequences of real numbers $\on$ with $|\on| \to + \infty$ and vectors $X_n = [u_n, v_n] \in \D(\A)$ with $\|X_n\|_\H = 1$ such that 
\begin{equation}
\label{eq:cont-Xn}
\omega_n^{2}  \| \mathcal{A}X_n + \mathrm{i} \omega_n X_n \|_\mathcal{H} \to 0,
\end{equation}
By taking a subsequence for which all $\on$ are either positive or all negative and replacing all $X_n$ by $-X_n$ if needed, we may assume that all $\on$ are positive. We shall obtain a contradiction by proving that $\|X_n\|_\H \to 0$ as $n$ goes to $+ \infty$. The proof is split into several steps as it involves some back and forths between estimates on $\Omega$, $\Gamma_0$, and $\Gamma_1$.

\textbf{Step 1: Obtaining Helmhotz-like equations.} We start by detailing \cref{eq:cont-Xn}:
\begin{subequations}
\label{eq:cont-Xn-detailed}
\begin{align}
\label{eq:v-u}
&\on^{2} (-v_n +\i \on u_n) \to 0 &&\mbox{in}~H^1(\Omega), \\
\label{eq:v-u-gamma}
&\on^{2} (-v_n +\i \on u_n) \to 0 &&\mbox{in}~H^1(\Gamma_0), \\
\label{eq:Dv-u}
&\on^{2} ( -\Delta u_n + \i \on v_n) \to 0 & &\mbox{in}~ L^2(\Omega), \\
\label{eq:Dv-u-gamma}
&\on^{2}( - \Delta_\Gamma u_n + \partial_\nu u_n + \i \on v_{n} ) \to 0 & & \mbox{in}~ L^2(\Gamma_0).
\end{align}
\end{subequations}
Plugging \cref{eq:v-u} into \cref{eq:Dv-u} and  \cref{eq:v-u-gamma} into \cref{eq:Dv-u-gamma} yields
\begin{subequations}
\label{eq:comp-vn}
\begin{align}
&\on(-\Delta u_n - \on^2 u_n) \to 0  &&\mbox{in}~L^2(\Omega),\\
&\on(-\Delta_\Gamma u_n -  \on^2 u_n + \partial_\nu u_n) \to 0 &&\mbox{in}~ L^2(\Gamma_0).
\end{align}
\end{subequations}
Let us reformulate \cref{eq:comp-vn} as follows: there exist sequences $\{f_n\} \subset L^2(\Omega)$ and $\{g_n\} \subset L^2(\Gamma_0)$ such that
\begin{subequations}
\label{eq:helmhotz}
\begin{align}
\label{eq:wave}
&- \Delta u_n - \on^2 u_n = f_n &&\mbox{in}~ \Omega,\\
\label{eq:wave-gamma}
&- \Delta_\Gamma u_n - \on^2 u_n = - \partial_\nu u_n + g_n &&\mbox{on}~ \Gamma_0,
\end{align}
\end{subequations}
with, using Landau notation,
\begin{equation}
\label{eq:est-rhs}
\|f_n\|_{L^2(\Omega)} = o(\on^{-1}) \quad \mbox{and}~ \|g_n\|_{L^2(\Gamma_0)} = o(\on^{-1}).
\end{equation}

\textbf{Step 2: Estimate of the feedback term.} Coming back to \cref{eq:cont-Xn}, we have
\begin{equation}
\label{eq:X-F}
\A X_n + \i \on X_n = F_n,
\end{equation}
where  $\{F_n \} \subset  \H$ is such that $\|F_n\|_\H = o(\on^{-2})$. Take the real part of the scalar product of \cref{eq:X-F} with $X_n$ to obtain
\begin{equation}
\label{eq:scalar}
\re (\A X_n, X_n)_\H = \re (F_n, X_n)_\H = o(\on^{-2}).
\end{equation}
Recalling the identity \cref{eq:ax-x}, it follows from \cref{eq:scalar} that
\begin{equation}
\label{eq:trace-vn}
\int_{\Gamma_1} |v_n|^2 \, \mathrm{d} \sigma = o(\omega_n^{-2}).
\end{equation}
Equation \cref{eq:trace-vn} together with \cref{eq:v-u} and continuity of the trace operator from $H^1(\Omega)$ to $L^2(\Gamma_1)$ yields $
\int_{\Gamma_1} |u_n|^2 \, \d \sigma = o(\on^{-4})$.
 Furthermore, since each $X_n$ is in $\D(\A)$, $\partial_\nu u_n + u_n = -\alpha v_n$ on $\Gamma_1$, and thus
\begin{equation}
\label{eq:est-feedback}
\int_{\Gamma_1} |\partial_\nu u_n|^2 \, \d \sigma = o(\on^{-2}).
\end{equation}

\textbf{Step 3: Estimate of the coupling term.}  It is assumed that $\overline{\Gamma_0} \cap \overline{\Gamma_1} = \emptyset$. As a consequence, there exists a vector field $\tilde{h} \in \mathcal{C}^2(\overline{\Omega})^d$ such that $\tilde{h} = \nu$ on $\Gamma_0$ and $\tilde{h} = 0$ on $\Gamma_1$. Multiplying \cref{eq:wave} by $2\tilde{h}\cdot \nabla \overline{u_n}$ and integrating over $\Omega$ leads to the following classical trace identity:
\begin{multline}
\label{eq:trace}
\int_{\Gamma_0} \omega_n^2 |u_n|^2 - \|\nabla u_n\|^2 \, \mathrm{d} \sigma   =  2 \re \int_\Omega [J_{\tilde{h}}\nabla u_n] \cdot  \nabla  \overline{u_n} \, \d x 
- \re   \int_{\Gamma_0} \partial_\nu u_n [2\tilde{h}\cdot \nabla \overline{u_n}] \, \d \sigma \\  - \re  \int_\Omega f_n [2\tilde{h}\cdot \nabla \overline{u_n} ]  \, \d x 
+ \int_\Omega \{ \omega_n^2 |u_n|^2 - \| \nabla u_n\|^2 \} \operatorname{div} \tilde{h} \, \mathrm{d}x,
\end{multline}
where $J_{\tilde{h}} = [\partial_j \tilde{h}_i]_{ij}$ is the Jacobian matrix of $\tilde{h}$. 
For the construction of $\tilde{h}$ or computations leading to \cref{eq:trace}, the reader is referred to  \cite[Lemmas 2.1 and 2.3]{komornik_exact_1994}.
Furthermore, since $\tilde{h} = \nu$ on $\Gamma_0$ and $\nu \cdot \nabla \overline{u_n} = \partial_\nu \overline{u_n}$,
\begin{equation}
\label{eq:h-para}
\re \int_{\Gamma_0} \partial_\nu u_n [2\tilde{h} \cdot \nabla \overline{u_n}] \, \d \sigma = 2 \int_{\Gamma_0} |\partial_\nu u_n|^2 \, \d \sigma.
\end{equation}
Now, recall that $\|X_n\|_\H = 1$. In particular,  $u_n$ and $v_n$ are bounded in $V$ and $H$ respectively, which implies:
\begin{itemize}
\item  $\nabla u_n$ and $\nabla_\Gamma u_n$  are bounded in $L^2(\Omega)^d$ and $L^2(\Gamma_0)^{d-1}$ respectively;
\item By \cref{eq:v-u}-\cref{eq:v-u-gamma}, $\on u_n$ is bounded in $L^2(\Omega)$ and $\on u_{n|\Gamma_0}$ is bounded in $L^2(\Gamma_0)$.
\end{itemize}
Therefore, it follows from \cref{eq:tan-riem}, \cref{eq:trace} and \cref{eq:h-para} that
\begin{equation}
\label{eq:est-coupling}
\int_{\Gamma_0} |\partial_\nu u_n|^2 \, \mathrm{d}\sigma = O(1).
\end{equation}

\textbf{Step 4: Multiplier identity.} Let $\eps > 0$ to be fixed later on and define
$
\M\overline{u_n} \triangleq 2h\cdot \nabla \overline{u_n} + (\dive h - \eps) \overline{u_n}
$,
where the vector field $h$ is defined in the hypotheses of Theorem \ref{th:pdr} and $\dive$ stands for the divergence. The multiplier identity
\begin{multline}
\label{eq:wave-mult}
2 \operatorname{Re} \int_\Omega [J_h \nabla u_n] \cdot \nabla \overline{u_n} \, \mathrm{d}x + \epsilon \int_\Omega \omega_n^2 |u_n|^2 - \|\nabla u_n\|^2 \, \mathrm{d}x \\ 
=  (1 - \eps) \operatorname{Re} \int_{\Gamma} \partial_\nu u_n \overline{u_n} \operatorname{div}h \, \mathrm{d}\sigma + \operatorname{Re} \int_{\Gamma} \partial_\nu u_n [2 h\cdot \nabla \overline{u_n}] \, \mathrm{d}\sigma \\ + \int_{\Gamma} (h\cdot \nu)  \{ \omega_n^2 |u_n|^2 - \| \nabla u_n\|^2\} \, \mathrm{d}\sigma +  \operatorname{Re} \int_\Omega f_n \mathcal{M}\overline{u_n} \, \mathrm{d}x.
\end{multline}
is standardly obtained by multiplying \cref{eq:wave} by $\M \overline{u_n}$, integrating over $\Omega$, and performing a series of integrations by parts -- see, e.g., the proof of \cite[Theorem 4.1]{lasiecka_uniform_1992} for similar computations.
Recalling Item \ref{it:jacobian} in the hypotheses, we choose $\eps < 2\rho$ to obtain
\begin{equation}
\label{eq:wave-mult-bis}
0 \leq (2 \rho - \eps) \int_{\Omega} \|\nabla u_n\|^2 \, \mathrm{d}x + \eps \int_\Omega \omega_n^2 |u_n|^2 \, \mathrm{d}x   \leq \mbox{Right-hand side of \cref{eq:wave-mult}.}
\end{equation}
In what follows, $\eta$ denotes an arbitrary number taken in $(0, 1)$. Since $\| \on u_n \|_H = O(1)$, we have
\begin{equation}
\label{eq:est-un}
\|u_n\|_{H} = o(\on^{\eta - 1}).
\end{equation}

\textbf{Step 5: Estimates on the boundary.} 
We start with the integrals on $\Gamma_0$.
{ First, $h\cdot \nu$ is smooth, so that $(h\cdot \nu)  \overline{u_n}$ belongs to $H^1(\Gamma_0)$ with} 
$
 \nabla_\Gamma [(h\cdot\nu)\overline{u_n}] = (h\cdot\nu) \nabla_\Gamma \overline{u_n} + \overline{u_n} \nabla_\Gamma [h\cdot \nu]
$.
Thus,  multiplying \cref{eq:wave-gamma} by $(h\cdot \nu) \overline{u_n}$, integrating over $\Gamma_0$, and using the Green formula \cref{eq:div-manifold} leads to
\begin{multline}
\label{eq:elliptic-gamma0}
\int_{\Gamma_0} (h \cdot \nu) \{ \|\nabla_\Gamma u_n\|^2 -  \omega_n^2 |u_n|^2\} \, \mathrm{d}\sigma  = \int_{\Gamma_0} (h \cdot \nu)  g_n \overline{u_n} \, \mathrm{d}\sigma \\ - \int_{\Gamma_0} (\nabla_\Gamma u_n \cdot \nabla_\Gamma [h\cdot \nu]) \overline{u_n} \, \d \sigma  - \int_{\Gamma_0} (h \cdot \nu)\partial_\nu u_n \overline{u_n} \, \mathrm{d}\sigma.
\end{multline}
Using a series of Cauchy-Schwarz inequalities, we deduce from \cref{eq:est-rhs}, \cref{eq:est-coupling}, \cref{eq:est-un}, and \cref{eq:elliptic-gamma0} that
\begin{equation}
\label{eq:est-gamma0}
\int_{\Gamma_0} (h \cdot \nu) \{ \|\nabla_\Gamma u_n\|^2 -  \omega_n^2 |u_n|^2\} \, \mathrm{d}\sigma  = o(\on^{\eta -1}).
\end{equation}
In view of \cref{eq:wave-mult}, we also note that because $h = (h\cdot \nu) \nu$ on $\Gamma_0$ (Item \ref{it:gamma0} in the hypotheses of Theorem \ref{th:pdr}), we have
\begin{equation}
\label{eq:tangential-vanish}
\re \int_{\Gamma_0} \partial_\nu u_n [2h\cdot \nabla \overline{u_n}] \, \d \sigma = 2\int_{\Gamma_0} (h\cdot \nu) |\partial_\nu u_n|^2 \, \d \sigma.
\end{equation}

Now we deal with the integrals on $\Gamma_1$ that appear in \cref{eq:wave-mult}. By using that $(h\cdot \nu) \geq m > 0$ on $\Gamma_1$ (Item \ref{it:gamma1}) together with Cauchy-Schwarz and Young inequalities, we get
\begin{multline}
\label{eq:cs-young-gamma1}
\int_{\Gamma_1} (h\cdot \nu)  \{ \omega_n^2 |u_n|^2 - \| \nabla u_n\|^2\} \, \mathrm{d}\sigma   + \operatorname{Re} \int_{\Gamma_1} \partial_\nu u_n [2 h\cdot \nabla \overline{u_n}] \, \mathrm{d}\sigma  \\ \leq  \int_{\Gamma_1} (h\cdot \nu)  \omega_n^2 |u_n|^2  + \frac{1}{2m} \int_{\Gamma_1} |\partial_\nu u_n|^2 \, \d \sigma.
\end{multline}

\textbf{Step 6: Estimate of the interior energy.} Bearing in mind the sign conditions prescribed for $(h \cdot \nu)$ on each part of $\Gamma$, we combine \cref{eq:wave-mult}, \cref{eq:wave-mult-bis}, \cref{eq:tangential-vanish}, and \cref{eq:cs-young-gamma1} to obtain
\begin{multline}
\label{eq:wave-mult-ter}
 (2 \rho - \eps) \int_{\Omega} \|\nabla u_n\|^2 \, \mathrm{d}x + \eps \int_\Omega \omega_n^2 |u_n|^2 \, \mathrm{d}x  \leq (1 - \eps) \operatorname{Re} \int_{\Gamma} \partial_\nu u_n \overline{u_n} \operatorname{div}h \, \mathrm{d}\sigma  \\ + \operatorname{Re} \int_\Omega f_n \mathcal{M}\overline{u_n} \, \mathrm{d}x  + \int_{\Gamma_0} (h\cdot \nu)  \{ \omega_n^2 |u_n|^2 - \| \nabla u_n\|^2\} \, \mathrm{d}\sigma \\
+  \int_{\Gamma_1} (h\cdot \nu)  \omega_n^2 |u_n|^2  + \frac{1}{2m} \int_{\Gamma_1} |\partial_\nu u_n|^2 \, \d \sigma.
\end{multline}
Next, we deduce from \cref{eq:wave-mult-ter} combined with \cref{eq:est-rhs}, \cref{eq:est-coupling}, \cref{eq:est-un}, and \cref{eq:est-gamma0} that
\begin{equation}
\label{eq:est-interior}
\int_\Omega \|\nabla u_n\|^2 + \on^2 |u_n|^2 \, \d x = o(\on^{\eta - 1}).
\end{equation}

\textbf{Step 7: Refined estimate of the coupling term.} We can now use \cref{eq:est-interior} to improve our prior estimate \cref{eq:est-coupling}. We come back to \cref{eq:trace} and \cref{eq:h-para}:
\begin{multline}
\label{eq:id-ref}
\int_{\Gamma_0} \omega_n^2 |u_n|^2 - \|\nabla_{\Gamma} u_n\|^2 + |\partial_\nu u_n|^2 \, \mathrm{d} \sigma    =   \int_\Omega \{ \omega_n^2 |u_n|^2 - \| \nabla u_n\|^2  \} \operatorname{div} \tilde{h} \, \mathrm{d}x \\  - \re  \int_\Omega f_n [2\tilde{h}\cdot \nabla \overline{u_n} ]  \, \d x 
+ 2 \re \int_\Omega [J_{\tilde{h}}\nabla u_n] \cdot  \nabla  \overline{u_n} \, \d x.
\end{multline}
As in Step 5, we obtain another expression of the integral over $\Gamma_0$ of $\on^2 |u_n|^2 - \|\nabla_\Gamma u_n\|^2$ by multiplying \cref{eq:wave-gamma} by $\overline{u_n}$ and integrating over $\Gamma_0$. Then, \cref{eq:id-ref} yields
\begin{multline}
\label{eq:normal-der}
\int_{\Gamma_0} |\partial_\nu u_n|^2 \, \d \sigma =   \int_\Omega \{  \omega_n^2 |u_n|^2 - \| \nabla u_n\|^2  \} \operatorname{div} \tilde{h} \, \mathrm{d}x   - \re  \int_\Omega f_n [2\tilde{h}\cdot \nabla \overline{u_n} ]  \, \d x   \\ + 2 \re \int_\Omega [J_{\tilde{h}}\nabla u_n] \cdot  \nabla  \overline{u_n} \, \d x  + \int_{\Gamma_0} g_n \overline{u_n} \, \d \sigma - \int_{\Gamma_0} \partial_\nu u_n \overline{u_n} \, \d \sigma.
\end{multline}
It follows from \cref{eq:est-rhs}, \cref{eq:est-un}, \cref{eq:est-interior}, and \cref{eq:normal-der} that
\begin{equation}
\label{eq:coupling-bis}
\int_{\Gamma_0} |\partial_\nu u_n|^2 \, \d \sigma = o(\on^{\eta - 1}).
\end{equation}

\textbf{Step 8: Conclusion.} Conversely, we can now use \cref{eq:coupling-bis} to refine the estimate \cref{eq:est-interior} of the interior energy. More precisely, using Cauchy-Schwarz inequality, we infer from \cref{eq:est-un} and \cref{eq:coupling-bis} that
\begin{equation}
\label{eq:refined-coupling}
\re \int_{\Gamma_0} \partial_\nu u_n \overline{u_n} \dive h = o(\on^{3(\eta - 1)/2})
\end{equation}
for any $\eta \in (0, 1)$. We let $\eta = 1/3$; then, by plugging \cref{eq:refined-coupling} into \cref{eq:wave-mult-ter} we finally obtain (compare with \cref{eq:est-interior})
\begin{equation}
\label{eq:est-interior-bis}
\int_\Omega \|\nabla u_n\|^2 + \on^2 |u_n|^2 \, \d x = o(\on^{-1}).
\end{equation}
We are now in position to conclude. We recall that the trace operator is continuous from $H^{1/2}(\Omega)$ into $L^2(\Gamma_0)$ and use linear interpolation between Sobolev spaces:
\begin{equation}
\label{eq:interp}
\begin{aligned}
\int_{\Gamma_0} \omega_n^2 |u_n|^2 \, \d \sigma& \leq K \omega_n^2 \|u_n\|^2_{H^{1/2}(\Omega)} \\
& \leq K' \omega_n \|u_n\|_{H^{1}(\Omega)}  \| \on u_n\|_{L^2(\Omega)}.
\end{aligned}
\end{equation}
By \cref{eq:est-interior-bis}, we have
\begin{equation}
\|u_n\|_{H^1(\Omega)} = o(\on^{-1/2}), \quad  \| \on u_n\|_{L^2(\Omega)} = o(\on^{-1/2}).
\end{equation}
Therefore, \cref{eq:interp} yields 
\begin{equation}
\int_{\Gamma_0} \on^2 |u_n|^2 \, \d \sigma = o(1)
\end{equation}
In sum, after multiplying \cref{eq:wave-gamma} by $\overline{u_n}$,
we finally obtain
\begin{multline}
\int_{\Omega} \omega_n^2 |u_n|^2 + \| \nabla u_n \|^2 \, \mathrm{d}x  + \int_{\Gamma_0} \omega_n^2 |u_n|^2 + \|\nabla_\Gamma u_n\|^2 \, \mathrm{d}\sigma  + \int_{\Gamma_1} |u_n|^2 \, \mathrm{d}\sigma = o(1),
\end{multline}
which contradicts $\|X_n\|_\mathcal{H} = 1$.
\end{proof}


\section*{Acknowledgements}

This work has been partially supported by MIAI@Grenoble Alpes (ANR-19-P3IA-0003).

\bibliographystyle{alpha}
\bibliography{cpde}

\begin{thebibliography}{{Tay}11}

\bibitem[{Ala}02]{alabau_indirect}
Fatiha {Alabau-Boussouira}.
\newblock {Indirect boundary stabilization of weakly coupled hyperbolic
  systems}.
\newblock {\em {SIAM Journal on Control and Optimization}}, 41(2):511--541,
  2002.

\bibitem[BD08]{batty-non-uniform}
Charles Batty and Thomas Duyckaerts.
\newblock Non-uniform stability for bounded semi-groups on {Banach} spaces.
\newblock {\em Journal of Evolution Equations}, 8(4):765--780, 2008.

\bibitem[BT10]{borichev_polynomial}
Alexander {Borichev} and Yuri {Tomilov}.
\newblock {Optimal polynomial decay of functions and operator semigroups}.
\newblock {\em {Mathematische Annalen}}, 347(2):455--478, 2010.

\bibitem[CST20]{chill-semi}
Ralph Chill, David Seifert, and Yuri Tomilov.
\newblock Semi-uniform stability of operator semigroups and energy decay of
  damped waves.
\newblock {\em Philosophical Transactions of the Royal Society A},
  378(2185):20190614, 2020.

\bibitem[GL14]{graber_analicity}
Philip~Jameson Graber and Irena Lasiecka.
\newblock Analyticity and {Gevrey} class regularity for a strongly damped wave
  equation with hyperbolic dynamic boundary conditions.
\newblock In {\em Semigroup Forum}, volume~88, pages 333--365. Springer, 2014.

\bibitem[GP10]{guillemin-topology}
Victor {Guillemin} and Alan {Pollack}.
\newblock {\em {Differential topology}}.
\newblock AMS Chelsea Publishing, 2010.

\bibitem[JS21]{jacob_stability}
Birgit Jacob and Nathanael Skrepek.
\newblock Stability of the multidimensional wave equation in port-{Hamiltonian}
  modelling.
\newblock In {\em 2021 60th IEEE Conference on Decision and Control (CDC)},
  pages 6188--6193, 2021.

\bibitem[Kom94]{komornik_exact_1994}
Vilmos Komornik.
\newblock {\em Exact controllability and stabilization: the multiplier method},
  volume~36.
\newblock Wiley, 1994.

\bibitem[Li21]{li_asymptotics}
Chan Li.
\newblock Asymptotics for wave equations with damping only on the dynamical
  boundary.
\newblock {\em Applied Mathematics and Optimization}, 84:1--16, 12 2021.

\bibitem[LL98]{liu1998spectral}
Bo~Liu and Walter Littman.
\newblock On the spectral properties and stabilization of acoustic flow.
\newblock {\em SIAM Journal on Applied Mathematics}, 59(1):17--34, 1998.

\bibitem[LM68]{lions-problemes}
Jacques-Louis Lions and Enrico Magenes.
\newblock {\em Probl\`emes aux limites non homog\`enes et applications},
  volume~1.
\newblock Dunod, 1968.

\bibitem[LR07]{liu_frequency_2007}
Zhuangyi Liu and Bopeng Rao.
\newblock Frequency domain approach for the polynomial stability of a system of
  partially damped wave equations.
\newblock {\em Journal of Mathematical Analysis and Applications},
  335(2):860--881, November 2007.

\bibitem[LT92]{lasiecka_uniform_1992}
Irena Lasiecka and Roberto Triggiani.
\newblock Uniform stabilization of the wave equation with {Dirichlet} or
  {Neumann} feedback control without geometrical conditions.
\newblock {\em Applied Mathematics and Optimization}, 25(2):189--224, March
  1992.

\bibitem[{Tay}11]{taylor-pde}
Michael~Eugene {Taylor}.
\newblock {\em {Partial differential equations I. Basic theory}}, volume 115.
\newblock Springer, 2011.

\bibitem[Vit17]{vitillaro2017wave}
Enzo Vitillaro.
\newblock On the wave equation with hyperbolic dynamical boundary conditions,
  interior and boundary damping and source.
\newblock {\em Archive for Rational Mechanics and Analysis}, 223(3):1183--1237,
  2017.

\end{thebibliography}

\end{document}